\documentclass[11pt,english,12p]{amsart}
\usepackage{amsmath}
\usepackage{amssymb}
\usepackage[foot]{amsaddr}
\usepackage[autostyle]{csquotes}
\usepackage{float}
\usepackage{tikz-cd} 
\usetikzlibrary{patterns}
\usetikzlibrary{matrix}
\usepackage{tikz}
\usepackage{amssymb,latexsym,amsmath,amsthm,amsfonts,graphicx}
\usepackage{verbatim}
\usepackage[autostyle]{csquotes}
\usepackage{mathrsfs}
\usepackage{float}
\usepackage{pst-node}
\usepackage{mathtools}
\usepackage{faktor}
\usepackage{hyperref}
\usepackage{orcidlink}
\usetikzlibrary{arrows.meta}

\usetikzlibrary{arrows,positioning,shapes,fit,calc}
\usetikzlibrary{decorations.pathreplacing}
\usepackage[all,cmtip]{xy}

\tikzstyle{littledot}=[circle, fill, inner sep=.4pt,minimum size=.4pt]

\newtheorem{thm}{Theorem}[section]
\newtheorem{lem}[thm]{Lemma}
\newtheorem{prop}[thm]{Proposition}
\newtheorem{cor}[thm]{Corollary}

\newtheorem{defn}[thm]{Definition}

\newtheorem{remark}[thm]{Remark}

\newcommand{\R}{{\mathbb R}}

\newcommand{\C}{{\mathbb C}}

\newcommand{\Pin}{\operatorname{Pin}}
\newcommand{\Spin}{\operatorname{Spin}}
\newcommand{\End}{\operatorname{End}}
\newcommand{\Pic}{\operatorname{Pic}}
\newcommand{\Aut}{\operatorname{Aut}}
\newcommand{\rk}{\operatorname{rk}}

\usepackage[automake, sort=use, nonumberlist]{glossaries}

\makeglossaries

\newglossaryentry{H}
{
    name=\ensuremath{H},
    description={a Hermitian metric on a complex vector space $V$}
}

\newglossaryentry{Clifford multiplication}
{
    name=\ensuremath{\hat{\rho}:\C_{q}(V)_{\mathbb{Z}}\rightarrow \End(S_{\Delta})},
    description={Clifford multiplication on our spinor torus $S_{\Delta}$}
}

\newglossaryentry{(V,q)}
{
    name=\ensuremath{(V,q)},
    description={a quadratic vector space with a form $q$ or $Q$}
}

\newglossaryentry{first Chern class}
{
    name=\ensuremath{c_1(L)},
    description={the first Chern class of a positive definite line bundle $L$}
}

\newglossaryentry{rho^f}
{
    name=\ensuremath{\rho^f},
    description={composition of Clifford multiplication with the adjoint conjugation of the isomorphism $f:S_{\Delta}\xrightarrow{\cong} 
E_{i}^{\times 2^{k}}$}
}
\newglossaryentry{Vector space}
{
    name=\ensuremath{V},
    description={a vector space over $\R$ or $\C$}
}
\newglossaryentry{PPAV}
{
    name=PPAV,
    description={a principally polarized Abelian variety}
}

\newglossaryentry{L}
{
    name=\ensuremath{L},
    description={a positive definite line bundle in $\Pic^{H}(V/\Gamma)$}
}

\newglossaryentry{Clifford group}
{
    name=\ensuremath{\Gamma_{q}(V)},
    description={the Clifford group of the Clifford algebra $\C_{q}(V)$}
}

\newglossaryentry{line bundles}
{
    name=\ensuremath{\Pic(S_{\Delta})},
    description={the variety of line bundles on $S_{\Delta}$}
}

\newglossaryentry{covering space}
{
    name=\ensuremath{T_{0}S_{\Delta}},
    description={the covering space for $S_\Delta$}
}

\newglossaryentry{E}
{
    name=\ensuremath{E},
    description={the alternating $(1,1)$ form that is the imaginary part of $H$ on $V$, $E = \mathrm{Im} H$}
}

\newglossaryentry{degree 0 line bundles}
{
    name=\ensuremath{\Pic^{0}(S_{\Delta})},
    description={the group of degree $0$ line bundles, vanishing $c_{1}(L_{\Delta})$}
}

\newglossaryentry{C(2^k)}
{
    name=\ensuremath{\mathbb{C}(2^k)},
    description={the matrix algebra of $2^{k}\times 2^{k}$ complex matrices}
}

\newglossaryentry{C_q(V)}
{
    name=\ensuremath{C_{q}(V)},
    description={the Clifford algebra of a quadratic vector space $V$ with a quadratic form $q$}
}

\newglossaryentry{(Delta, H)}
{
    name=\ensuremath{(\Delta,H)},
    description={a unitary spinor module, where $H$ is the positive definite Hermitian form associated with the chosen anti-involution $*$}
}

\newglossaryentry{principal polarization}
{
    name=\ensuremath{L_{\Delta}},
    description={the principal polarization for $S_{\Delta}$}
}

\newglossaryentry{spinor Abelian variety}
{
    name=\ensuremath{S_{\Delta}},
    description={the spinor Abelian variety associated to the spinor module $\Delta$}
}

\newglossaryentry{half spinor}
{
    name=\ensuremath{\Delta^{+}} and \ensuremath{\Delta^{-}},
    description={the half spinor modules associated with $\Delta$}
}

\newglossaryentry{complex torus}
{
    name=\ensuremath{V\slash \Gamma},
    description={a complex torus formed by the quotient of $V$ by a discrete lattice $\Gamma$}
}

\newglossaryentry{Gamma}
{
    name=\ensuremath{\Gamma},
    description={a lattice of full rank in a complex vector space $V$}
}

\newglossaryentry{complexification}
{
    name=\ensuremath{\mathbb{C}_{q}(V)},
    description={the complexification of $C_{q}(V)$}
}

\newglossaryentry{Delta}
{
    name=\ensuremath{\Delta},
    description={a unitary spinor module for the Clifford algebra $\C_{q}(V)$}
}

\newglossaryentry{C_2k}
{
    name=\ensuremath{\mathbb{C}_{2k}},
    description={the complexification of the Clifford algebra associated to the quadratic space $\R^{2k}$ of signature $(0,2k)$}
}

\newglossaryentry{multiplicative generators, real}
{
    name=\ensuremath{\hat{\Gamma}_{q}(V)},
    description={the finite group of multiplicative generators of the Clifford algebra $C_{q}(V)$}
}

\newglossaryentry{Dirac spinor Abelian variety}
{
    name=\ensuremath{S_{\Delta_{2k}}},
    description={the Dirac spinor Abelian variety}
}

\newglossaryentry{Gamma_Delta}
{
    name=\ensuremath{\Gamma_\Delta},
    description={a full rank lattice in $\Delta$}
}

\newglossaryentry{product of elliptic curves}
{
    name=\ensuremath{E^{\times 2^{k}}_{i}},
    description={the product of $2^{k}$ copies of the elliptic curves $E_{i}=\dfrac{\C}{\mathbb{Z}\oplus i\cdot \mathbb{Z}}$}
}

\newglossaryentry{integer subring}
{
    name=\ensuremath{\mathbb{C}_{q}(V)_{\mathbb{Z}}},
    description={the integer subring of $\mathbb{C}_{q}(V)$}
}

\newglossaryentry{multiplicative generators}
{
    name=\ensuremath{\hat{\Gamma}^{c}_{q}(V)},
    description={the finite group of multiplicative generators of the Clifford algebra $\C_{q}(V)$}
}

\renewcommand*{\glossarymark}[1]{}

\selectfont

\title[Clifford multiplication on Spinor Abelian Varieties]{Clifford Multiplication on Spinor Abelian Varieties}

\author[I.\ Grzegorczyk and R.\ Su\'arez]{Ivona Grzegorczyk and Ricardo Su\'arez}
\address{Department of Mathematics, California State University Channel Islands; Natural Science Division, Pepperdine University}

\email{ivona.grzegorczyk@csuci.edu}
\email{josericardo.suarez@pepperdine.edu}


\begin{document}

\normalsize

\begin{abstract} 
We define a spinor Abelian variety $S_{\Delta}$ to be a complex Abelian variety whose tangent space at the origin is a space of spinors for a suitable complex Clifford algebra $\mathbb{C}_{q}(V)$. We examine intrinsic properties of such varieties and 
the connection between Clifford multiplication and their endomorphism algebras. We then extend the analysis of Clifford multiplication to the dual torus $\Pic^{0}(S_{\Delta})$. 

\end{abstract}
\maketitle

\section{Introduction} 

Clifford algebras provide a unified mathematical framework for geometry and physics, and hence there is an increased interest in studying them in various contexts. They are crucial in constructing spinors that are used in describing elementary particles or in unifying Maxwell's equations.  Spinors play a major role as a tool in detecting parity changes when looking for hidden symmetries (supersymmetries) of spaces,  in quantum mechanics or general relativity.  The concept of algebraic spinors was introduced several decades ago by Chevalley and Cartan, who described their algebraic and geometric properties in \cite{Car,Ch}. An interesting result by Satake  in  \cite{Sat} links families of Abelian varieties with the even subalgebras of real Clifford algebras of total signatures greater than or equal to $2$  where $(p,q)\neq (1,1)$ by providing a suitable complex structure on these varieties  (usually given by a subalgebra generator whose square is negative). Using this idea, for any Clifford algebra of dimension $n$, we can generate a complex torus of dimension $2^{n-2}$. 

In this paper, we link complex Clifford algebras with certain Abelian varieties by constructing suitable complex spinor spaces. We focus on Abelian varieties obtained as quotients  \gls{complex torus} of a vector space \gls{Vector space} by   a full-rank lattice \gls{Gamma}, satisfying some Clifford algebra conditions. In Section \ref{Spinor Abelian Varieties}, we define Abelian varieties constructed this way as \textit{spinor Abelian varieties}, denoted by \gls{spinor Abelian variety}, and associated to the complex Clifford algebra $\C_{q}(V)$ with a complex spinor module \gls{Delta} (where $\Delta$ is a spinor space for the Clifford algebra). We  describe Clifford actions (called multiplications) on these varieties and study their properties. In Proposition \ref{dual spinor AV}, we show that for any spinor Abelian variety, its dual variety $\Pic^{0}(S_{\Delta})$ is also a spinor Abelian variety. Thus the Clifford multiplication associated with the spinor Abelian variety $S_{\Delta}$ also induces Clifford multiplication on $\Pic^{0}(S_{\Delta})$. Moreover, generators of the Clifford algebra $\C_{q}(V)$ are now in bijection with line bundles $L\in \Pic^{0}(S_{\Delta})$ that are either trivial or have the properties $L^{\otimes 2}\cong \mathcal{O}_{S_{\Delta}}$ or $L^{\otimes 4}\cong\mathcal{O}_{S_{\Delta}}$.

 We describe some intrinsic properties of spinor Abelian varieties resulting from an understanding of their endomorphism structure. For example, Lemma \ref{Losing hat} (\emph{Losing your hat lemma})  links Clifford multiplication, the representations of the associated Clifford algebra, and the analytic representations of $S_{\Delta}$. 
 
 In Theorem \ref{full decomposition}, using the structure of the endomorphism ring of our spinor Abelian varieties, we conclude that they are fully decomposable as the direct sum of $2^{k}$ (the dimension of $S_{\Delta}$) copies of an elliptic curve of $j$-invariant $1728$, which we denote $E_{i}$. As an immediate consequence, we can state that  $E_{i}^{\times 2^{k}}$ is itself a spinor Abelian variety with Clifford multiplication on $E_{i}^{\times 2^{k}}$ by $\C_{q}(V)_{\mathbb{Z}}$ induced from Clifford multiplication on $S_{\Delta}$. 

\printglossary[title=List of Symbols]

\section{Background on Abelian varieties and  Clifford algebras}
In this section, we start by providing some background in both Abelian varieties and Clifford algebras required to properly define spinor Abelian varieties and Clifford multiplication on them. For more information about complex Abelian varieties, see \cite{BL,GH}; and see \cite{LM,Me} for  properties of Clifford algebras.

\subsection{Complex Abelian varieties} 
\begin{defn}
Let $V$ be a finite-dimensional complex vector space. A \textbf{Hermitian metric} (or a positive definite Hermitian form) \gls{H} is a complex bi-additive map, $H:V\times V\rightarrow \C$, with the following properties:
\begin{enumerate}
\item $H$ is complex linear in the first argument. 
\item $H$ has conjugate symmetry, that is,   $H(v,w)=\overline{H(w,v)}$ for all $v,w\in V$.
\item $H$ is a positive definite real-valued quadratic form on $V$, where $H(v,v)\ge 0$ and $H(v,v)\in\R$  for all $v\in V$.

\end{enumerate}
A finite-dimensional complex vector space $V$ with a Hermitian metric $H$ is called a \textbf{Hermitian (or unitary) vector space}. 
\end{defn}
 For any Hermitian metric on $V$, the imaginary part, which we denote by \gls{E} (i.e.\ $E=\mathrm{Im} H$), is a real skew-symmetric form on $V$.
 
\begin{defn}
Let $V$ be a finite-dimensional complex vector space. A \textbf{lattice} $\Gamma$ in $V$ is a discrete subgroup such that the quotient $V/\Gamma$ is compact. That is, $\Gamma$ is a free Abelian group of full rank, i.e.\  $\rk\ \Gamma=\dim_{\R} V$. The quotient $V/\Gamma$ of the complex vector space $V$ by the lattice $\Gamma$ is called a \textbf{complex torus}.
    	
\end{defn}

\begin{defn}
 A complex torus $V/\Gamma$ is an \textbf{Abelian variety} if there exists a positive definite Hermitian form $H$ on $V$ such that the imaginary part $E=\mathrm{Im}  H$ of the Hermitian form is integral on the lattice $\Gamma\subset V$. Then the pair $(V/\Gamma,H)$ is called a  \textbf{polarized Abelian variety}.

\end{defn}
The following remark provides alternative, but equivalent, ways to define a polarization on a complex torus.

\begin{remark}\label{Polarization conditions}
    
\rm{1. One may also define a polarization on $V/\Gamma$ as a first Chern class $\gls{first Chern class}=H$ of a positive definite line bundle $\gls{L}\in \Pic^{H}(V/\Gamma)$, relating the positive definite Hermitian form on $V$ with our polarization. 

2. Alternatively, we can define a polarization as an alternating form $E:\Gamma\times\Gamma\rightarrow \mathbb{Z}$ acting on the lattice $\Gamma $ such that it gives an extension to real scalars, i.e.\ $\Gamma\otimes \R=V$, which is defined as $E: V\times V\rightarrow \R$ where  $E(iv,iw)=E(v,w)$ and $E(iv,v)>0$. These conditions are known as the\textbf{ Riemann relations}, and when Riemann relations are  satisfied by an alternating $(1,1)$ form $E$,  we obtain a related polarization on the Abelian variety.}
\end{remark}


The skew-symmetric form $E$  giving us  our polarization can be defined in some basis  $\gamma_1,\ldots ,\gamma_{2g}$ as a skew-symmetric matrix,  $E=\begin{array}{c} \left( \begin{array}{cc} 0_ {g\times g}& D \\ -D& 0_{g\times g}\\ \end{array} \right), \\ \end{array}$ where the diagonal matrix $ D=\operatorname{diag}\ (d_1,\ldots ,d_g)\in \mathbb{Z}_{\ge 0}^{g}$, and where the entries are ordered by the relation $d_i| d_{i+1}$.  This way the sequence $(d_1,\ldots ,d_g)$ is unique and defines a skew-symmetric form up to an isomorphism. Hence, the sequence  $D$  is called the type of the polarization
(see \cite[p.\ 70]{BL},\cite[p.\ 306]{GH}).

\begin{defn}
	Let $V$ be a finite-dimensional complex vector space. An Abelian variety $V/\Gamma$ with the polarization form $E$ is said to be \textbf{principally polarized} if the polarization type of $E$  is given by  $D=I_{g\times g}$. Equivalently, $V/\Gamma$ is a principally polarized Abelian variety if $det(E)=1$, for the form $E$ defining  the polarization of our Abelian variety. An Abelian variety with a principal polarization is called a \textbf{principally polarized Abelian variety}, which we denote by \gls{PPAV} hereafter. 
\end{defn}
Note that elliptic curves are PPAVs of dimension one over $\C$, and in this paper we consider  elliptic curves admitting complex multiplication.

\begin{defn}
An elliptic curve is said to have \textbf{complex multiplication} if its endomorphism ring  $\End(E)$ is strictly greater than $\mathbb{Z}$. 
\end{defn}
The elliptic curve  defined by the lattice spanned by $1$ and $i$, denoted $E_{i}=\dfrac{\C}{\mathbb{Z}\oplus i\cdot \mathbb{Z}}$, has the Gaussians  as its endomorphism ring (that is, $\End(E_{i})=\mathbb{Z}[i]$), and its automorphism group is the multiplicative group generated by $i\in\mathbb{C}$. When it comes to endomorphisms, we have that for any complex torus $V/\Gamma$ of dimension $n$, the endomorphism ring $\End_{\mathbb{Z}}(V/\Gamma)$  is a free $\mathbb{Z}$-module with the property that $\rk(\End_{\mathbb{Z}}(V/\Gamma))\leq 2 n^2$. When the endomorphism ring is of full rank, we have the following proposition (see \cite{Sh}).

\begin{prop}\label{Decomposition}
Let $V/\Gamma$ be a complex torus of dimension $n$. If the rank of the endomorphism ring is $2 n^2$, then $V/\Gamma$ is isogenous to the direct sum of $n$ copies of an elliptic curve $E$ with complex multiplication.
\end{prop}

For complex Abelian varieties, we have two standard representations for the endomorphism ring, induced from the fact that any  endomorphism $f\in \End(V/\Gamma)$ is given by a $\C$-linear map from $V$ to itself, such that its restriction to the lattice $\Gamma$  is contained in the lattice. This prompts the following definition.

\begin{defn}
   Let $V/\Gamma$ be a polarized Abelian variety with the endomorphism ring  $\End(V/\Gamma)$. $\End(V/\Gamma)$ induces two injective ring homomorphisms: 
   \begin{enumerate}
   \item $\tau_{a}:\End(V/\Gamma)\rightarrow \End_{\C}(V)\cong \C(dim V)$ given by $\tau_{a}(f)=f_{a}$, and \item $\tau_{r}:\End(V/\Gamma)\rightarrow \End_{\mathbb{Z}}(\Gamma)\cong \mathbb{Z}(2\cdot dim V)$ given by $\tau_{r}(f)=f_{r}$. 
   \end{enumerate}
   $\tau_{a}$ is called the \textbf{analytic representation}, and  $\tau_{r}$ the \textbf{rational representation}.
\end{defn}

The following proposition states that for any PPAV, the property of having certain automorphisms provides us with information about its full decomposition into products of elliptic curves.  

\begin{prop}\label{Lange thm}
    Suppose that $f\in \Aut(V/\Gamma)$ is an automorphism of order $d\ge 3$ with $\tau_{a}(f)=\zeta_{d}\cdot id_{V}$, where $\zeta_{d}$ is a $d$-th root of unity. Then $d\in \{3,4,6\}$, and $V/\Gamma\cong E\oplus \cdots\oplus E=:E^{\oplus \dim V}$, where $E$ denotes the elliptic curve admitting automorphisms of order $d$.
 \end{prop}

{\it{Proof:}} See (\cite[p.\ 420]{BL}).
\begin{remark}
    \rm{ One can conclude that if, for some automorphism $f$, the matrix representation that defines  the analytic representation $\tau_{a}(f)$ in $\C(\dim V)$ is of the form $i\cdot I_{\dim V}\in \C(\dim V)$, then $V/\Gamma$ fully decomposes as a product of elliptic curves of $j$-invariant $1728$. It is through these analytic representations that we link up the right type of complex Abelian varieties (which we later call spinor Abelian varieties) with an associated Clifford algebra.  }

\end{remark}

\subsection{Clifford algebras}
Let $V$ be a vector space of real dimension $n$. From now on, we use signatures $(p,q)$ for our quadratic forms that define  quadratic structures on $V$, where $p+q=n$. 
The $q$ in the signature refers to the number of negative definite generators. Note that  when we refer to a quadratic vector space \gls{(V,q)}, the $q$ symbolizes the quadratic form associated with the quadratic vector space. 
While this may be confusing, this is a standard notation used in the case of Clifford algebras.

\begin{defn}

Let $(V,q)$ be a quadratic vector space over $\R$, where the form $q$ is of signature $(p,q)$. Let $V^{\otimes}$ be the tensor algebra associated to $(V,q)$. We define the \textbf{ideal generated by $q$} as $I_{q}=\langle v\otimes v-q(v)1_V: v\in V \rangle$. The \textbf{Clifford algebra} associated to the  quadratic vector space $(V,q)$ is the quotient $ \gls{C_q(V)}=V^{\otimes}/I_{q}$. For any real quadratic space $(V,q)$, we denote by $\C_{q}(V)$ the natural complexification of the Clifford algebra; that is, $\C_{q}(V)=C_{q}(V)\otimes_{\R} \C$.
\end{defn}

We denote the $k$-th graded component of any element $u\in C_{q}(V)$ in the Clifford algebra by $\langle u\rangle_{k}=\sum_{I\subset [n]: |I|=k} u_{I}e_{I}$, where $I=(j_1,\ldots,j_{k})$ with $1\leq j_1< \cdots <j_k\leq n$ and $[n]=\{1,2, \ldots ,n\}$,  $u_{I}\in \C$, and $e_{I}$ is the Clifford product of basis elements of the form $e_{j_1} \cdots e_{j_k}$. Thus the Clifford algebra $C_{q}(V)$ is a $\mathbb{Z}_2$-graded super algebra; that is, $C_{q}(V)=C^{+}_{q}(V)\oplus C^{-}_{q}(V)$ where $C^{+}_{q}(V)$ is the even subalgebra consisting of elements of an even bi-degree and  $C^{-}_{q}(V)$ is the odd part associated to the $\mathbb{Z}_2$ grading.

We now  define several important subgroups of $C_{q}(V)$ that we use in this paper.

\begin{defn}
We denote by $C_{q}(V)^{*}$ the group of invertible elements of the Clifford algebra. The \textbf{Clifford group}, denoted $\Gamma_{q}(V)$, is the subgroup of $C_{q}(V)^{*}$ that preserves $V$ under the adjoint action; that is,  $\gls{Clifford group}=\{g\in C_{q}(V)^{*}: gvg^{-1} \in V\}$. If we restrict the Clifford group  to the even invertible elements, we have what is called  the special Clifford group $\Gamma^{+}_{q}(V)=\Gamma_{q}(V)\cap C^{+}_{q}(V)^{*}$. The subgroup of $\Gamma_{q}(V)$ generated by elements $v\in V$ with $q(v)=\pm 1$ is called the \textbf{Pin group}. That is, $\Pin(V,q)=\{v_{1} \cdots v_{k}\in \Gamma_{q}(V):q(v_{j})=\pm 1\}.$ The \textbf{Spin group} is the subgroup of the Pin group generated by elements of an even grade, defined as $\Spin(V,q)=\{v_{1} \cdots v_{2m}\in \Pin(V,q)\}.$ Lastly, choosing  an orthonormal basis $e_{1},\ldots,e_{n}$, for the vector space $V$, we denote the finite subgroup of the multiplicative generators of $C_{q}(V)$ by $\gls{multiplicative generators, real}=\{ \pm e_{I}:I\subset [n]\}$.
\end{defn}

For the complexification $\C_{q}(V)$ we have the following definition. 

\begin{defn}
For the complexification $\C_{q}(V)$, we denote  the \textbf{complexified Spin groups} by $\Gamma^{c}_{q}(V)$, $\Spin(V_{\C})$, and $\Pin(V_{\C})$, along with the multiplicative group of generators \gls{multiplicative generators}.
\end{defn}

The complexified $\Spin$ groups contain the original $\Spin$ groups of $C_{q}(V)$. Moreover, we can view the multiplicative group of generators $\hat{\Gamma}_{q}^{c}(V)$ as the group $\hat{\Gamma}_{q}(V)\times \langle i \rangle $. This is because if we view the generators  of the algebra as a real basis, we have  the generators $1, e_{I}$, along with $i, i e_{I}$ for the imaginary generators, where we generate  $-1$ and $-i$ by products between the generators. 






    

We now shift our attention to unitary spinor modules, which are used in the construction of our special complex Abelian varieties.  We begin by fixing a special type of  involution on the Clifford algebra  $\C_{q}(V)$.
\begin{defn}
We define $*$ to be any conjugate antilinear involution on the complex Clifford algebra $\C_{q}(V)$ 
which satisfies the following: 
\begin{itemize}
\item $(u\cdot v)^{*}=v^{*}\cdot u^{*}$, for any $u,v\in\C_{q}(V)$.
\item $(cu)^{*}=\bar{c}u^{*}$, for any $u\in\C_{q}(V)$ and $c\in\C$.

\end{itemize}

\end{defn}

\begin{remark}
On $\C_{q}(V)$ we have a conjugate linear anti-automorphism $u^{*}=\tilde{\bar{u}}$, where $(u\cdot v)^{*}=v^{*}\cdot u^{*}$, and $(cu)^{*}=\bar{c}u^{*}$.  Here $\bar{u}$ is the extension of conjugation on the complex vector space $V\otimes \C$, and $\tilde{u}$ is the reversion anti-automorphism in $C_{q}(V)$.
\end{remark}

If, additionally, $*$ defines inverses for the Clifford algebra $\C_{q}(V)$, then the finite multiplicative group of complex generators $\hat{\Gamma}^{c}_{q}(V)$ sits comfortably inside the infinite group $\Pin_{c}(V)=\{x\in\Gamma(V\otimes \C):x^{*}x=1\}$. (See \cite{Me} for more on these groups). We now define unitary spinor modules for our Clifford algebra $\C_{q}(V)$.

\begin{defn}
For the complex Clifford algebra $\C_{q}(V)$, a \textbf{unitary spinor module} with respect to the anti-linear involution $*$ is a Hermitian super vector space \gls{(Delta, H)} with an isomorphism of  algebras 
\[\rho:\C_{q}(V)\xrightarrow{\cong} \End(\Delta)\]
such that for any $g\in \C_{q}(V)$, we have $\rho(g^{*})=\rho(g)^{*}$. The involution $*$ on $\End(\Delta)$, coming from the anti-linear involution on the Clifford algebra $\C_{q}(V)$, is the adjoint operation determined by our unique Hermitian metric $H$. 
\end{defn}


\begin{prop}
Any spinor module $\Delta$ admits a Hermitian metric, unique up to positive scalars, for which it becomes a unitary spinor module.

\end{prop}
\noindent \textit{Proof:} See \cite[p.\ 78]{Me}.

Note that in the case of unitary spinor modules $\Delta$, the restriction of the unitary algebra isomorphism to the Spin groups preserves the Hermitian inner product $H$ on $\Delta$. That is, $\rho_{g}^{*}H=H$ for any element $g$ belonging to one of the Spin groups. When $*$ is defined to give us inverses for any choice of basis for $\C_{q}(V)$, then any $v\in V\otimes \C$ gives us a self-adjoint operator $\rho(v)\in \End(\Delta)$.

\section{Spinor Abelian Varieties}\label{Spinor Abelian Varieties} 

In this section, we consider a unitary spinor module $\Delta$ for an even-dimensional complexified  Clifford algebra $\C_{q}(V)$ (where $V$ is a real vector space of dimension $2k$), and we assume that $\Delta$ is a Hermitian vector space with a  Hermitian metric form defined on it.
We introduce  the following description of the spinor torus.
\subsection{Clifford multiplication on spinor tori }

\begin{defn}
Consider an even-dimensional complexified  Clifford algebra $\C_{q}(V)$. We define the quotient of its unitary spinor module $\Delta$ by a full rank lattice $\gls{Gamma_Delta}\subset \Delta$ as the associated  \textbf{spinor torus}, which we denote by $S_{\Delta}=\Delta/\Gamma_{\Delta}$.

\end{defn}

\begin{remark}
	\rm{For odd-dimensional vector spaces, with $\dim_{\C} V=2k+1$, we can obtain  spaces of spinors  of the form  $\Delta^{+}\oplus \Delta^{-}$, using the representations $\rho:\C_{q}(V)\xrightarrow{\cong} \End(\Delta^{+})\oplus \End(\Delta^{-})$.  Note that in this case \gls{half spinor} are of the same dimension, $2^{k}$, and are known as half spinor spaces. Since these half spinor spaces are just spinor spaces for a Clifford  algebra $\C_{q}(V)$ for some vector space of complex dimension $2k$, we mainly deal with even-dimensional cases for $V$.}
\end{remark}
We can view the  spinor torus $S_{\Delta}$ as a complex torus whose covering space  $\gls{covering space}=\Delta$ is a unitary spinor module associated to a Clifford algebra $\C_{q}(V)$ of some quadratic vector space. Hence $T_{0}S_{\Delta}$ satisfies the property that its space of endomorphisms is isomorphic, as a complex algebra, to the associated Clifford algebra; that is,  $\C_{q}(V)\cong \End(T_{0}S_{\Delta})$.
 We need all of the above to   define Clifford multiplication on a spinor torus properly, as a reduction of the isomorphism between the Clifford algebra and the space of endomorphisms.
\begin{defn}
We define $\C_{q}(V)_{\mathbb{Z}}$ as the full rank lattice associated with the complexified Clifford algebra $\C_{q}(V)$, when we view $\C_{q}(V)$ as a dimension $2^{2k}$ complex vector space.

\end{defn}
Any element $h\in \C_{q}(V)_{\mathbb{Z}}$ may be referred to as a lattice element of $\C_{q}(V)$. It should also be noted that $\C_{q}(V)_{\mathbb{Z}}$, as a full rank lattice, is an Abelian subgroup under addition, and multiplication in the algebra is closed and distributes across addition; that is, $\C_{q}(V)_{\mathbb{Z}}$ is  itself a subring of the Clifford algebra $\C_{q}(V)$. We can view the integer subring 
\gls{integer subring}
in a few equivalent but different ways. If we view $C_{q}(V)\subset \C_{q}(V)$ as the real form of the complex Clifford algebra and restrict its scalars to $\mathbb{Z}$-linear combinations, which we denote by $C_{q}(V)_{\mathbb{Z}}$,  we have a full rank lattice in $C_{q}(V)$. We can then $\mathbb{Z}$-tensor this lattice with the Gaussian ring $\mathbb{Z}[i]$ to obtain its extension as a lattice (or integer subring) onto $\C_{q}(V)$; that is, $\C_{q}(V)_{\mathbb{Z}}=C_{q}(V)_{\mathbb{Z}}\otimes_{\mathbb{Z}} \mathbb{Z}[i]$. Another approach is to choose  a complex  basis for $\C_{q}(V)$, say $e_{I}\otimes 1$ for a given basis $e_{I}$, in the real form $C_{q}(V)$, allowing us to view $\C_{q}(V)$ as a  dimension $2^{2k}$ complex vector space. Using this basis, we chose a real and imaginary basis, viewing our vector space as a real vector space of real  dimension $2^{2k+1}$. That is, we give it the real basis $e_{I}\otimes 1,e_{I}\otimes i$, where $e_{I}:I\subset [n]$  is the basis of the real form $C_{q}(V)$ and $I\subset [n]$ is an increasing sequence. Both approaches do require us to define a basis. The third approach just takes into account that the real quadratic vector space $(V,q)$ is a $\mathbb{Z}$-module under the operation of addition, $(V,+)$, if we restrict our scalar field from $\R$ to $\mathbb{Z}$; we denote this $\mathbb{Z}$-module, or its full rank lattice, by $V_{\mathbb{Z}}$. We then denote its tensor algebra (as $\mathbb{Z}$-tensors) by $V_{\mathbb{Z}}^{\otimes_{\mathbb{Z}}}$. Taking its quotient by the two-sided ideal obtained by restricting the quadratic form on $V$ to $V_{\mathbb{Z}}$, which we denote  by $I_{q}^{\mathbb{Z}}$, we obtain the ring  $C_{q}(V)_{\mathbb{Z}}=(V_{\mathbb{Z}})^{\otimes_{\mathbb{Z}}}/I^{\mathbb{Z}}_{q}$ over $\mathbb{Z}$. We then extend this natural construction by taking a $\mathbb{Z}$-tensor with the Gaussians to define the lattice $\C_{q}(V)_{\mathbb{Z}}$. We now define Clifford multiplication on our spinor torus as the restriction of the algebra isomorphism to this full rank lattice.

\begin{defn}
\textbf{Clifford multiplication on the spinor torus $S_{\Delta}$} is given as a descension of the unitary representation isomorphism  $\rho: \C_{q}(V)\xrightarrow{\cong} \End(\Delta)$ to a $\mathbb{Z}$-module homomorphism \gls{Clifford multiplication}. Clifford multiplication on our spinor torus $S_{\Delta}$ is then defined as the endomorphisms on $S_{\Delta}$ associated to the full rank lattice $\C_{q}(V)_{\mathbb{Z}}$ of the Clifford algebra $\C_{q}(V)$.

\end{defn}
Note that the full rank lattice $\Gamma_{\Delta}$ is chosen so that when we restrict the isomorphism $\rho$ to  $\hat{\rho}$, then  $\Gamma_{\Delta}$ is preserved for all lattice points $h\in\C_{q}(V)_{\mathbb{Z}}$.
This choice of lattice does depend, therefore, on both $\rho$ and $\Gamma_{\Delta}$ in a way that allows our isomorphism to descend.  We may use the term \textit{lattice actions} to refer to Clifford multiplication on $S_{\Delta}$. Also, the lattice actions on $S_{\Delta}$ restricted to the multiplicative group of generators $\hat{\Gamma}^{c}_{q}(V)\subset \C_{q}(V)_{\mathbb{Z}}$ give  us a finite group action on the spinor torus $S_{\Delta}$. The above definition of Clifford multiplication on our spinor tori $S_{\Delta}$ requires a closer look. When we define a basis, we can consider our full rank lattice as a direct sum $\C_{q}(V)_{\mathbb{Z}}=C_{q}(V)_{\mathbb{Z}}\oplus i\cdot C_{q}(V)_{\mathbb{Z}}$ where $C_{q}(V)_{\mathbb{Z}}$ is the integer subring of the real Clifford algebra, $C_{q}(V)$, pre-complexification. We provide the diagram in Figure \ref{football} for clarification of what we mean by the integer subring $\C_{q}(V)_{\mathbb{Z}}\subset \C_{q}(V)$.

	
	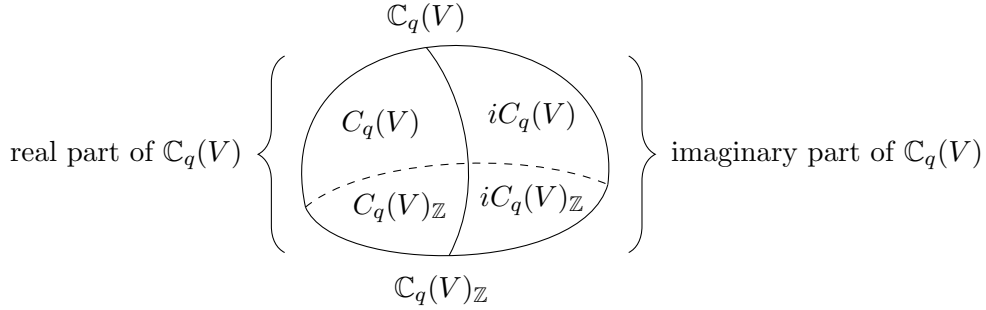
\begin{figure}[h!]
		\centering
		
		\begin{tikzpicture}[scale=1]
			\draw (-2,-.8) .. controls (-2.6,2) and (2.3,2) .. (2,-0.5);
			
			\draw (-2,-0.8) .. controls (-1.6,-1.7) and (1.65,-1.7) .. (2,-0.5);
			
			\draw [dashed] (-2,-0.8) .. controls (-1.3,-.2) and (0.8,0) .. (2,-0.5);
			
			\draw (-.4,1.31) .. controls (0.1,.7) and (0.4,-.5) .. (-.1,-1.44);
			
			\draw [decorate,
			decoration = {brace, raise=5pt,amplitude=10pt}] (-2.1,-1.4) -- (-2.1,1.2);
			
			\node [left] at (-2.7,-.1) {real part of $\mathbb{C}_q(V)$};
			
			\draw [decorate,
			decoration = {brace, raise=5pt,amplitude=10pt}] (2.1,1.2) -- (2.1,-1.4);
			
			\node [right] at (2.7,-.1) {imaginary part of $\mathbb{C}_q(V)$};
			
			\node at (-.4,1.7) {$\mathbb{C}_q(V)$};
			
			\node at (-.2,-1.9) {$\mathbb{C}_q(V)_\mathbb{Z}$};
			
			\node at (-1,0.3) {$C_q(V)$};
			
			\node at (1,0.4) {$i C_q(V)$};
			
			\node at (-.75,-.8) {$C_q(V)_\mathbb{Z}$};
			
			\node at (1,-.7) {$i C_q(V)_\mathbb{Z}$};
			
		\end{tikzpicture}
		
		\caption{$\C_{q}(V)_{\mathbb{Z}}$ as a lattice on $\C_{q}(V)$ and its structure in relation to $C_{q}(V)$}
	\label{football}	
	\end{figure}

Notice that Clifford multiplication as defined preserves the full rank lattice $\Gamma_{\Delta}\subset \Delta$, and the restriction to the integer Spin groups preserves the Hermitian metric on $\Delta$ and the full rank lattice in our spinor torus $S_{\Delta}$. Note that some of our restrictions needed  to define Clifford multiplication and the structure of our spinor torus may be dropped. 
For instance, we  may look for complex tori satisfying  the property of having only $C_{q}(V)_{\mathbb{Z}}$ actions but not $\C_{q}(V)_{\mathbb{Z}}$ actions. This is equivalent to stating that the imaginary part of $\C_{q}(V)_{\mathbb{Z}}$ does not preserve the lattice  $\Gamma_{\Delta}$. But since this lattice would be preserved only  by $C_{q}(V)_{\mathbb{Z}}$ actions, it is only multiplication by $i$ that is the problem (as it would not preserve the lattice). More specifically, restricting Clifford multiplication so that the actions are  only from the lattice of the real form, $C_{q}(V)_{\mathbb{Z}}\subset C_{q}(V)$, is equivalent to saying $i\cdot \Gamma_{\Delta}\not \subset \Gamma_{\Delta}$, but $C_{q}(V)_{\mathbb{Z}}\cdot \Gamma_{\Delta}\subset \Gamma_{\Delta} $. A spinor torus that satisfies this criteria has a covering space that is a space of complex spinors for the real form $C_{q}(V)$ but not the complexification, and hence is of a different type than the ones we have discussed above. This additional restriction requires its own definition. 
	

\begin{defn}
A spinor torus  that admits only  $C_{q}(V)_{\mathbb{Z}}$ multiplication (but does not admit $\C_{q}(V)_{\mathbb{Z}}$ multiplication)  is called a \textbf{strictly real spinor torus}, and we denote it by $S_{\Delta^{\R}}$ to distinguish it from $S_{\Delta}$. On a strictly real spinor torus, Clifford multiplication comes from the restriction $\rho^{\R}:C_{q}(V)_{\mathbb{Z}}\rightarrow \End (S_{\Delta^{\R}})$, where $\Delta^{\R}$ is still a complex vector space that defines a representation for the real  Clifford algebra $C_{q}(V)$ and not its complexification $\C_{q}(V)$.

\end{defn}
\begin{remark}
\rm{Although we do not provide examples of strictly real spinor tori, we do want to bring attention to their structure. The way we would construct strictly real spinor tori is to begin with a real spinor module $\Delta^{\R}$  and define a complex structure $J:\Delta^{\R}\rightarrow \Delta^{\R}$, making $\Delta^{\R}$ a complex vector space with a full rank lattice preserved by Clifford multiplication.}

\end{remark}
From the preceding discussion, for any spinor torus $S_{\Delta}$ we have $i\cdot \Gamma_{\Delta}\subset \Gamma_{\Delta}$; hence any of the generators $e_{I}\in \hat{\Gamma}_{q}(V)$ of order $2$ or $4$ acting on $S_{\Delta}$ also has a complex action, given by $i\cdot e_{I}$, which is of order $4$ or $2$ respectively.
\subsection{Spinor Abelian varieties and some elementary properties} 
In this subsection we work with spinor tori $S_{\Delta}$  that have the additional structure of being principally polarized. We introduce the following definition.


\begin{defn}
Let $S_{\Delta}$ be a spinor torus with Clifford multiplication such that the positive definite Hermitian form $H$ on $\Delta$ defines a principal  polarization for $S_{\Delta}$. Then $S_{\Delta}$ is called a \textbf{spinor Abelian variety}.
\end{defn}

For any spinor Abelian variety constructed from a unitary spinor module $\Delta$ associated with a complex Clifford algebra $\C_{q}(V)$ of complex dimension $2^{2k}$ with a positive definite Hermitian form $H$, we  call $H$  its polarization, hence making $S_{\Delta}$ a complex  Abelian variety of dimension $2^{k}$. When we have the additional structure of $H$ defining a principal polarization, then the spinor Abelian variety is categorized as a PPAV of dimension $2^{k}$ (as is often the case).

\begin{remark}
\rm{We may have Clifford modules that generate tori with Clifford multiplication  that fail to be spinor Abelian varieties (they also fail to be spinor tori). For example, $\C_{q}(V)$ is itself a complex Clifford module, via left multiplication, and is itself a unitary Clifford module. However, $\C_{q}(V)$ (although a Clifford module) is not a spinor space since $\End(\C_{q}(V))$ is not the same space as $\C_{q}(V)$. Hence the space of endomorphisms is bigger in this case, and therefore there is no isomorphism. Moreover, one may construct  a Clifford module with no notion of a polarization in mind.  Thus we conclude that at the bare minimum, the ring of endomorphisms of the covering space of the spinor Abelian variety must be of complex dimension $2^{2k}$. Since  $\End(\C_{q}(V))$ is of higher dimension, $\C_{q}(V)$ modulo  a full rank lattice will not give us a spinor Abelian variety.}

\end{remark}

We now carefully analyze Clifford multiplication on the spinor torus $S_{\Delta}$. We start with the following lemma that provides a description of  the  lattice actions $\C_{q}(V)_{\mathbb{Z}}$ on $S_{\Delta}$ in terms of the translation holomorphisms  $t_{x}:S_{\Delta}\rightarrow S_{\Delta}$, given by $t_{x}y=x+y$ for all $x,y\in S_{\Delta}$.

\begin{lem}\label{translation constants}
For any lattice element $h\in \C_{q}(V)_{\mathbb{Z}}$ and $\bar{\lambda}\in S_{\Delta}$, there exists an element $\bar{\mu}\in S_{\Delta}$ such that Clifford multiplication by $h$  on $S_{\Delta}$ is represented by  translation by $\bar{\mu}$; that is, $\rho_{h}(\bar{\lambda})=t_{\bar{\mu}}(\bar{\lambda})$.

\end{lem}
\begin{proof}
Fix an element $\bar{\lambda}\in S_{\Delta}$ and a lattice point $h\in\C_{q}(V)_{\mathbb{Z}}$. 
We can now define $\bar{\mu}_{\bar{\lambda},h}=\rho_{h}(\bar{\lambda})-\bar{\lambda}\in S_{\Delta}$. Clearly $\bar{\mu}_{\bar{\lambda},h}$ is an element in $S_{\Delta}$, as  $S_{\Delta}$ is by definition a complex Abelian  Lie group with  group operation given by addition. Hence, we can compute $t_{\bar{\mu}_{\bar{\lambda},h}}(\bar{\lambda})=\bar{\lambda}+(\rho_{h}(\bar{\lambda})-\bar{\lambda})=(\bar{\lambda}-\bar{\lambda})+\rho_{h}(\bar{\lambda})=\rho_{h}(\bar{\lambda})$.
\end{proof}

As a consequence of Lemma \ref{translation constants},  we can formulate the following definition.

\begin{defn}
Consider any $\bar{\lambda}\in S_{\Delta}$ and lattice point $h\in\C_{q}(V)_{\mathbb{Z}}$. We define $M_{\bar{\lambda},h}\in S_{\Delta}$ as the \textbf{translation element associated with the action} $\rho_{h}(\bar{\lambda})$ if $t_{M_{\bar{\lambda},h}}(\bar{\lambda})=\rho_{h}(\bar{\lambda})$.

\end{defn}
The above means that we can consider Clifford multiplication endomorphisms on our spinor torus in terms of translations. The following proposition provides some  insight into the translation elements given by generators of the Clifford algebra acting on the principally polarized spinor torus $S_{\Delta}$.

\begin{prop}\label{translation thm}
Consider a spinor Abelian variety $S_{\Delta}$.  Then for any $\bar{\lambda}\in S_{\Delta}$ and generator $e_{I}\in\Gamma_{q}(V)$ of order $4$, we have a  system of translation elements $M,N\in S_{\Delta}$ satisfying  $ \bar{\lambda}^{-1}=\dfrac{1}{2} (M+N)$ such that  
\[   \left\{
\begin{array}{ll}
   \rho_{e_{I}}(\bar{\lambda})=t_{M}(\bar{\lambda})     \\
      \rho^{2}_{e_{I}}(\bar{\lambda})=t_{M+N}(\bar{\lambda})\\
       \rho^{3}_{e_{I}}(\bar{\lambda})=t_{N}(\bar{\lambda})\\
\rho^{4}_{e_{I}}(\bar{\lambda})=t_{0}(\bar{\lambda}).\\
\end{array} 
\right. \]
\end{prop}

\begin{proof}
Fix any generator $e_{I}\in\Gamma_{q}(V)$  of  order $4$ and $\bar{\lambda}\in S_{\Delta}$.  Then by Lemma \ref{translation constants}, we have  $\rho_{e_{I}}(\bar{\lambda})=\bar{\lambda}+M$ for some translation element $M\in S_{\Delta}$ associated with the lattice action by $e_{I}$ and the element $\bar{\lambda}$.  By repeating this process, we get $\rho^2_{e_{I}}(\bar{\lambda})=-\bar{\lambda}$, as well as the equation $\rho^2_{e_{I}}(\bar{\lambda})=(\bar{\lambda}+M)+N$ for some translation element $N\in S_{\Delta}$ associated with the lattice action by $e_{I}$ and the element $\bar{\lambda}+M$. Using the above two translation equations, we can write $\bar{\lambda}+M+N=-\bar{\lambda}$. Now,  solving for $-\bar{\lambda}=\bar{\lambda}^{-1}$, we get the equation   $ \bar{\lambda}^{-1}=\dfrac{1}{2} (M+N)$. By composing the action with itself for a third time, we get $\rho_{e_{I}}^3(\bar{\lambda})=-\rho_{e_{I}}(\bar{\lambda})=-(M+\bar{\lambda})$. Moreover,  we also have $\rho^3_{e_{I}}(\bar{\lambda})=(\bar{\lambda}+M+N)+O$ for some translation element $O\in S_{\Delta}$ associated with the lattice action by $e_{I}$ and the element $\bar{\lambda}+M+N$. Setting both equations for $\rho^{3}_{e_{I}}(\bar{\lambda})$ together and solving for $\bar{\lambda}^{-1}$  yields the equation $\bar{\lambda}^{-1}=M+\dfrac{1}{2}(N+O)$. When we substitute this expression for  $\bar{\lambda}^{-1}$ with   $ \bar{\lambda}^{-1}=\dfrac{1}{2} (M+N)$, we get the equality  $O=-M$. Hence we obtain  $\rho^{3}_{e_{I}}(\bar{\lambda})=\bar{\lambda}+M+N+O=\bar{\lambda}+M+N-M=\bar{\lambda}+N=t_{N}(\bar{\lambda})$. Note that we also  have $e_{I}^4=1$. Therefore $\rho^4_{e_{I}}(\bar{\lambda})=id(\bar{\lambda})=t_0(\bar{\lambda})$.

\end{proof}
It follows from Proposition \ref{translation thm} that for any generator, we need two associated translation constants $M$ and $N$ (associated with $e_{I}$ and $\bar{\lambda}$) to generate all orders of Clifford multiplication of $\bar{\lambda}\in S_{\Delta}$ by a given lattice point $e_{I}\in \Gamma_{q}(V)$ in terms of the associated translation. This  prompts the following definition. 

\begin{defn}
For any generator $e_{I}\in\Gamma_{q}(V)$ and element $\bar{\lambda}\in S_{\Delta}$, we define the translation elements  $M,N$ that define all orders of  Clifford multiplication on $\bar{\lambda}$ by a lattice point $e_{I}\in \Gamma_{q}(V)$ as  \textbf{the  Clifford translation elements  $M,N$ associated to multiplication by the lattice point $e_{I}$.}

\end{defn}
In the case of the 2-torsion points, which we denote  $J_{2}^{S_{\Delta}}$,  we have the following corollary.

\begin{cor}\label{2 torsion corrollary}
Consider a 2-torsion point $\epsilon\in J_{2}^{S_{\Delta}}\subset S_{\Delta}$. The  actions by generators $e_{I}\in\Gamma_{q}(V)$ of any order greater than one yields one translation element $M$ which is itself a 2-torsion point.
\end{cor}

\begin{proof}
Fix any generator $e_{I}\in\Gamma_{q}(V)$  of any  order greater than one. Also fix a 2-torsion point $\epsilon\in J_{2}^{S_{\Delta}}\subset S_{\Delta}$ as in Proposition \ref{translation thm} satisfying the equation $ {\epsilon}^{-1}=\dfrac{1}{2} (M+N)$ for the translation elements $M$ and $N$ associated with the action of $e_{I}$. For 2-torsion points we have  $\epsilon^{-1}=\epsilon$. Hence we get the equation $2\epsilon=0=M+N$. Then it follows that the  second translation element associated with the action of $e_{I}$ on $\epsilon$ is $M^{-1}$. To prove that $M$ is itself a 2-torsion point, we use the linearity property associated with the endomorphism $\rho_{e_{I}}:S_{\Delta}\rightarrow S_{\Delta}$, where  by Lemma \ref{translation constants} we have $2\cdot \rho_{e_{I}}(\epsilon)=2(\epsilon+M)$. Using the  bilinearity property, we also have  the equation $2\cdot \rho_{e_{I}}(\epsilon)=\rho_{e_{I}}(2\cdot \epsilon)=\rho_{e_{I}}(0)=0$. Hence we obtain the equation $2\cdot (\epsilon + M)=0$, which immediately  implies $2\epsilon +2 M=0$. Therefore  $2M=0$, forcing $M$ to be a 2-torsion point on $S_{\Delta}$. Moreover, it immediately follows that $N=M^{-1}=M$.

\end{proof}
At this time, we can conclude that $\hat{\Gamma}^{c}_{q}(V)$ actions on $J^{S_{\Delta}}_2$ can be described in terms of the induced translation morphisms. We quickly remark that by the nature of the 2-torsion points, any action by a lattice point in $\C_{q}(V)_{\mathbb{Z}}$ reduces to an action by a generator in $\hat{\Gamma}^{c}_{q}(V).$

We now extend these properties into the dual Abelian variety of $S_{\Delta}$, defined as  $\gls{degree 0 line bundles}=\{L\in \gls{line bundles}:c_1(L)=0\}$ (see [8] for more on the dual lattice). We start with the following proposition.

\begin{prop}\label{dual spinor AV}
For any spinor Abelian variety $S_{\Delta}$, the group $\Pic^{0}(S_{\Delta}) $ of line bundles with a vanishing first Chern class is also  a spinor Abelian variety.

\end{prop}

\begin{proof}
Let  $S_{\Delta}$ be a spinor Abelian variety  for the Clifford algebra $\C_{q}(V)_{\mathbb{Z}}$. Then $S_{\Delta}$ is a PPAV with Clifford multiplication given by $\hat{\rho}:\C_{q}(V)_{\mathbb{Z}}\rightarrow \End(S_{\Delta})$. One can  also define the principal polarization of $S_{\Delta}$ as a positive definite line bundle $\gls{principal polarization}\in \Pic^{H}(S_{\Delta})=\{L\in \Pic(S_{\Delta}):c_1(L)=H\}$ whose first Chern class is  $c_1(L_{\Delta})=H$, where $H$ is the positive definite Hermitian form on $\Delta$. 
Then the  principal polarization $L_{\Delta}$ defines an isomorphism $\phi_{L_{\Delta}}:S_{\Delta}\xrightarrow{\cong} \Pic^{0}(S_{\Delta})$ between $S_{\Delta}$ and $\Pic^{0}(S_{\Delta})$, given by  $\phi_{L_{\Delta}}(\bar{\lambda})=t^{*}_{\bar{\lambda}}L_{\Delta}\otimes L_{\Delta}^{-1}$ for any $\bar{\lambda}\in S_{\Delta}$, where $t^{*}_{\bar{\lambda}}:\Pic(S_{\Delta})\rightarrow \Pic(S_{\Delta})$ is the pullback of the line bundles in the Picard variety along the translation morphism $t_{\bar{\lambda}}:S_{\Delta}\rightarrow S_{\Delta}$  (see \cite[p.\ 36]{BL},\cite[p.\ 10]{Do}). Via this isomorphism, we can easily conclude that $\Pic^{0}(S_{\Delta})$ is a PPAV, where the required polarization on $\Pic^{0}(S_{\Delta})$ is given by the inverse isomorphism $\phi_{L_{\Delta}}^{-1}:\Pic^{0}(S_{\Delta})\rightarrow S_{\Delta}$, and the principal polarization is defined by $(\phi_{L_{\Delta}}^{-1})^{*}L_{\Delta}$. Now, to show that $\Pic^{0}(S_{\Delta})$ is a spinor Abelian variety, we need to properly define Clifford multiplication on it.  We first state that by the surjectivity of the isomorphism $\phi_{L_{\Delta}}:S_{\Delta}\xrightarrow{\cong} \Pic^{0}(S_{\Delta})$, for every class $M\in \Pic^{0}(S_{\Delta})$ we have a class $\hat{\mu}\in S_{\Delta}$ such that $\phi_{L_{\Delta}}(\hat{\mu})=t^{*}_{\bar{\mu}}L_{\Delta}\otimes L_{\Delta}^{-1}=M$. Hence we have  the equation $\phi^{-1}_{L_{\Delta}}(M)=\bar{\mu}$. By using the inverse of the isomorphism induced by the above polarization,  we  can extend Clifford multiplication onto  $\Pic^{0}(S_{\Delta})$  via $\rho^{*}:\C_{q}(V)_{\mathbb{Z}}\rightarrow \End(\Pic^{0}(S_{\Delta}))$, where $\rho^{*}=Ad_{\phi_{L_{\Delta}}}\circ \hat{\rho}$. That is, for any lattice point  $h\in \C_{q}(V)_{ \mathbb{Z}}$, we have the following diagram: 
\[
\begin{tikzcd}
S_{\Delta} \arrow[r, "\hat{\rho}_{h}"]
& S_{\Delta} \arrow[d, "\phi_{L_{\Delta}}"] \\
\Pic^{0}(S_{\Delta}) \arrow[u, "\phi_{L_{\Delta}}^{-1}"] \arrow[r, red, "{\rho}^{*}_{h}" blue]
&  \Pic^{0}(S_{\Delta}).
\end{tikzcd}
\]
This means that for any line bundle $M\in \Pic^{0}(S_{\Delta})$, we have $\rho^{*}_{h}(M)=\phi_{L_{\Delta}}\circ \hat{\rho_{h}} \circ \phi^{-1}_{L_{\Delta}}(M)$. With the induced Clifford multiplication on $\Pic^{0}(S_{\Delta})$, we conclude that $\Pic^{0}(S_{\Delta})$ is a PPAV with Clifford multiplication on the underlying  dual torus, hence  a spinor Abelian variety.
\end{proof}

Now, considering $\Pic^{0}(S_{\Delta})$ as a spinor Abelian variety, we have the immediate consequence that  the principal polarization on $S_{\Delta}$  is preserved by the integer Spin groups. 

\begin{cor}\label{translation thm Line bundles}
On the dual spinor Abelian variety $\Pic^{0}(S_{\Delta})$, consider any $L\in \Pic^{0}(S_{\Delta})$ and any generator $e_{I}\in\Gamma_{q}(V)$ of order $4$. Then we have  a  system of translation line bundles  $L_{M},L_{N}\in \Pic^{0}(S_{\Delta})$ satisfying  $ (L^{\vee})^{\otimes 2}=L_{M}\otimes L_{N}$ such that  
\[   \left\{
\begin{array}{ll}
   \rho^{*}_{e_{I}}(L)=L\otimes L_{M}     \\
      (\rho_{e_{I}}^{*})^{2}(L)=L\otimes L_{M} \otimes L_{N} \\
       (\rho_{e_{I}}^{*})^{3}(L)=L\otimes L_{N}  \\
(\rho_{e_{I}}^{*})^{4}(L)=L\otimes \mathcal{O}_{S_{\Delta}}\cong L \\
\end{array} 
\right. \]

Hence  any generator of order $4$ acting on a line bundle $L_{\bar{\lambda}}\in \Pic^{0}(S_{\Delta})$ generates the Clifford system of line bundles $\{ L_{M},L_{M}\otimes L_{N},L_{N},\mathcal{O}_{S_{\Delta}}\}$.
\end{cor}

\begin{proof}
Fix a generator $e_{I}\in\hat{\Gamma}_{q}(V)$ of order $4$ and a line bundle in the Picard group $L\in \Pic^{0}(S_{\Delta})$ such that under the isomorphism $\phi_{L_{\Delta}}$ induced by the principal polarization $L_{\Delta}$, the preimage of this line bundle is in some class $\bar{\lambda}\in S_{\Delta}$ such that $\phi_{L_{\Delta}}(\bar{\lambda})=L$. As we saw in the proof of Proposition \ref{dual spinor AV}, Clifford multiplication can be defined as $\rho^{*}_{e_{I}}(L)=\phi_{L_{\Delta}}\circ \rho_{e_{I}}\circ\phi^{-1}_{L_{\Delta}} (L).$ Then we have 
$$\rho^{*}_{e_{I}}(L)=\phi_{L_{\Delta}}\circ \rho_{e_{I}}(\bar{\lambda})=\phi_{L_{\Delta}}(\bar{\lambda}+M)=\phi_{L_{\Delta}}(\bar{\lambda})\otimes \phi_{L_{\Delta}}(M)=L\otimes L_{M},$$ where we define  $L_{M}:=\phi_{L_{\Delta}}(M)$. Now, composing this action with itself, we obtain $ (\rho_{e_{I}}^{*})^{2}(L)=\phi_{L_{\Delta}}\circ \rho^2_{e_{I}}\circ\phi^{-1}_{L_{\Delta}} (L)=\phi_{L_{\Delta}}\circ \rho^2_{e_{I}}(\bar{\lambda})=\phi_{L_{\Delta}}(-\bar{\lambda})=L^{\vee}$. Considering this same action from a different perspective, we obtain
\begin{align*}
(\rho_{e_{I}}^{*})^{2}(L) & =\rho^{*}_{e_{I}}(L\otimes M) = \phi_{L_{\Delta}}\circ \rho_{e_{I}}(\bar{\lambda}+M) \\
& = \phi_{L_{\Delta}}(\bar{\lambda}+M+N)  = \phi_{L_{\Delta}}(\bar{\lambda})\otimes \phi_{L_{\Delta}}(M)\otimes \phi_{L_{\Delta}}(N) \\
& = L\otimes L_{M}\otimes L_{N},
\end{align*}
where we define  $L_{N}:=\phi_{L_{\Delta}}(N)$. Considering both of the above expressions for $  (\rho_{e_{I}}^{*})^{2}(L)$, we obtain the equation $L\otimes L_{M}\otimes L_{N}\cong L^{\vee}$. This gives  us the line bundle equation $(L^{\vee})^{\otimes 2}=L_{M}\otimes L_{N}$. Continuing this process,   we get 
\[ (\rho_{e_{I}}^{*})^{3}(L)=\phi_{L_{\Delta}}\circ \rho^3_{e_{I}}\circ\phi^{-1}_{L_{\Delta}} (L)=\phi_{L_{\Delta}}\circ \rho^3_{e_{I}}(\bar{\lambda})=\phi_{L_{\Delta}}(-(\bar{\lambda}+M))=L^{\vee}\otimes {M}^{\vee}.\] 
Once again, if we view  this same action from a different perspective, we obtain
\begin{align*}
(\rho_{e_{I}}^{*})^{3}(L) & =\rho^{*}_{e_{I}}(L\otimes L_{M}\otimes L_{N})=\phi_{L_{\Delta}}\circ \rho_{e_{I}}(\bar{\lambda}+M+N) \\ & = \phi_{L_{\Delta}}(\bar{\lambda}+M+N+O)=\phi_{L_{\Delta}}(\bar{\lambda})\otimes \phi_{L_{\Delta}}(M)\otimes \phi_{L_{\Delta}}(N)\otimes \phi_{L_{\Delta}}(O) \\
& =L\otimes L_{M}\otimes L_{N}\otimes L_{O},
\end{align*}
where we define  $L_{O}:=\phi_{L_{\Delta}}(O)$. Considering  both expressions for $  (\rho_{e_{I}}^{*})^{3}(L)$, we obtain the  equation $L\otimes L_{M}\otimes L_{N}\otimes L_{O}\cong L^{\vee}\otimes L_{M}^{\vee}$. Hence we have  $(L^{\vee})^{\otimes 2}\cong  L_{M}^{\otimes 2}\otimes L_{N}\otimes L_{O}$. Now taking both expressions for $(L^{\vee})^{\otimes 2}$, we get  $L^{\otimes 2}\otimes L_{N}\otimes L_{O}\cong L_{M}\otimes L_{N}$. Thus we have   $L_{O}\cong L_{M}^{\vee}$, providing us with the conclusion 
\[(\rho_{e_{I}}^{*})^{3}(L)\cong L\otimes L_{M}\otimes L_{N}\otimes L_{O}\cong L\otimes L_{M}\otimes L_{N}\otimes L^{\vee}_{M}\cong L\otimes L_{N}.\] Continuing this procedure, one can easily deduce that $ (\rho_{e_{I}}^{*})^{4}(L)\cong L\otimes \mathcal{O}_{S_{\Delta}}\cong L.$ Since our choice of a line bundle and  a generator of order $4$ were completely arbitrary, we conclude that for  any generator of order $4$, the Clifford system of line bundles $\{ L_{M},L_{M}\otimes L_{N},L_{N},\mathcal{O}_{S_{\Delta}}\}$ is associated to each subsequent action on $L$.

\end{proof}

From the preceding discussion, we see that $L_{M},L_{N}$, and $L_{O}$ depend on  $L\in \Pic(S_{\Delta})$ and the endomorphisms associated to the generator $e_{I}$; hence these equations are dependent on the choice of generator $e_{I}$ and  $L\in \Pic^{0}(S_{\Delta})$. Thus we introduce the following definition.

\begin{defn}
For any generator $e_{I}\in\hat{\Gamma}_{q}(V)$ and any line bundle  $L\in \Pic^{0}(S_{\Delta})$, we define the translation bundles   $L_{M},L_{N}$ (i.e., as above, line bundles defining all orders of  Clifford multiplication on $L$ by a lattice point $e_{I}\in \Gamma_{q}(V)$)  as  \textbf{the  Clifford line bundles associated to multiplication by a lattice point $e_{I}$.}
\end{defn}

Now we extend this property for  points of order 2 onto $\Pic^{0}(S_{\Delta})_2$.

\begin{cor}\label{2-torsion line bundle cor}
Consider the subgroup of line bundles of order $2$, $\Pic^{0}(S_{\Delta})_2=\{L\in \Pic^{0}(S_{\Delta}):L^{\otimes 2}\cong \mathcal{O}_{S_{\Delta}}\}$. The  actions by any generator $e_{I}\in\Gamma_{q}(V)$ of any order greater than one yields one Clifford translation bundle $L_{M}$, which is itself a line bundle of order $2$.
\end{cor}

\begin{proof}
Choose a generator $e_{I}\in\Gamma_{q}(V)$  of any  order greater than one. Choose a line bundle of order $2$, i.e.\  $L\in \Pic^{0}(S_{\Delta})_2$. Now by Corollary \ref{translation thm Line bundles}, we can write  $(L^{\vee})^{\otimes 2}\cong L_{M}\otimes L_{N}$ for the Clifford translation line bundles $L_{M},L_{N}$ associated with  the action of $e_{I}$ on $L$. Since $L\in \Pic^{0}(S_{\Delta})_2$, we have $L^{\vee}=L$, and hence  we can immediately deduce that $(L^{\vee})^{\otimes 2}\cong  L^{\otimes 2}\cong \mathcal{O}_{S_{\Delta}}\cong L_{M}\otimes L_{N}$. Therefore we have $L_{N}\cong  L^{\vee}_{M}$. Taking the induced representation of $L \otimes L$, we get $\rho^{*}(L^{\otimes 2})=\phi_{L_{\Delta}}\circ \rho_{e_{I}}\circ \phi^{-1}_{L_{\Delta}}(L\otimes L)=\phi_{L_{\Delta}}\circ \rho_{e_{I}}(2\bar{\lambda})=\phi_{L_{\Delta}}(2\rho_{e_{I}}(\lambda))=\phi_{L_{\Delta}}(\bar{\lambda})^{\otimes 2}=(L\otimes M)^{\otimes 2}=L^{\otimes 2}\otimes L_M^{\otimes 2}\cong \mathcal{O}_{S_{\Delta}}\otimes L_M^{\otimes 2}\cong L_ M^{\otimes 2}.$ But also, since $L\in \Pic^{0}(S_{\Delta})_2$, we have $L^{\otimes 2}\cong \mathcal{O}_{S_{\Delta}}$,  so that  $\rho^{*}_{e_{I}}(L^{\otimes 2})\cong\rho^{*}_{e_{I}}(\mathcal{O}_{S_{\Delta}})\cong \mathcal{O}_{S_{\Delta}}$. Now by setting both expressions for $\rho^{*}(L^{\otimes 2})$ equal to one another, we obtain $L_M^{\otimes 2}\cong\mathcal{O}_{S_{\Delta}}$. This forces $M$ to be a line bundle of order $2$. Moreover, $L_{N}=L_{M}^{\vee}$. Hence $L_{N}=L_{M}$. Therefore we conclude that each action by a Clifford generator only generates one Clifford translation bundle $L_{M}$, which is itself a line bundle of order $2$.

\end{proof}

Having established some of the elementary properties of our spinor Abelian varieties, we now shift our focus to some immediate intrinsic properties revealed by the study of their endomorphism rings.

\subsection{The endomorphism structure of spinor Abelian varieties}
In this section we examine the endomorphism  ring of our spinor Abelian variety $S_{\Delta}$ of dimension $2^{k}$.  The following lemma examines the relationship between the analytic representations, Clifford multiplication, and the spinor representations, through what we call the ``losing your hat lemma''.

\begin{lem}[Losing your hat lemma]\label{Losing hat}
For the spinor Abelian variety  $S_{\Delta}$, the analytic representation $\tau_{a}: \End(S_{\Delta})\rightarrow \End(\Delta)$ satisfies the property $\tau_{a}(\hat{\rho}(h))=\rho(h)$ for any $h\in\C_{q}(V)_{\mathbb{Z}}$.
    
\end{lem}

\begin{proof}

For any spinor Abelian variety, Clifford multiplication  $\hat{\rho}:\C_{q}(V)_{\mathbb{Z}}\rightarrow \End(S_{\Delta})$ is the ring homomorphism obtained by the restriction of the isomorphism $\rho:\C_{q}(V)\xrightarrow{\cong} \End(\Delta)$. Now if we fix an element $h\in\C_{q}(V)_{\mathbb{Z}}$, the endomorphism $\hat{\rho}(h)\in \End(S_{\Delta})$ can be viewed as the restriction  $\rho\big|_{\C_{q}(V)_{\mathbb{Z}}}(h)\in \End(S_{\Delta})$. Thus it is clear that the spinor representation $\rho$ defines Clifford multiplication on $S_{\Delta}$ by any lattice element $h\in\C_{q}(V)_{\mathbb{Z}}$. Then we can view the analytic endomorphism $\tau_{a}:\End(S_{\Delta})\rightarrow \End(\Delta)$ for any endomorphism of the form $\hat{\rho}(h)$, for a lattice element $h\in\C_{q}(V)_{\mathbb{Z}}$, as  $\tau_{a}(\hat{\rho}(h))=\rho(h)$. This is because for any endomorphism $\hat{\rho}(h)\in \End(S_{\Delta})$, there is an endomorphism $\rho(h)\in \End(\Delta)$ that defines it.  
    
\end{proof}

Thus for Clifford multiplication $\hat{\rho}$, composing it with the analytic representation $\tau_{a}$ gives us $\rho$, losing the hat on Clifford multiplication and providing us with the following commutative diagram.

$$\begin{tikzcd}
 & \End(S_{\Delta})\arrow{dr}{\tau_{a}} \\
\C_{q}(V)_{\mathbb{Z}} \arrow{ur}{\hat{\rho}} \arrow{rr}{\tau_{a}\circ \hat{\rho}=\rho} && \End(\Delta)
\end{tikzcd}$$.

With Lemma \ref{Losing hat}, we are able  to prove the following proposition. 

\begin{prop}\label{endomorphism isomorphism}
For a spinor Abelian variety $S_{\Delta}$ with Clifford multiplication given by   $\C_{q}(V)_{\mathbb{Z}}$-lattice actions, we have the ring isomorphism  $\C_{q}(V)_{\mathbb{Z}}\cong \End(S_{\Delta})$.
    
\end{prop}

\begin{proof}
For spinor Abelian varieties, we have Clifford multiplication given by the ring homomorphism $\hat{\rho}:\C_{q}(V)_{\mathbb{Z}}\rightarrow \End(S_{\Delta})$. 
Now suppose that for $h,g\in\C_{q}(V)$, $\hat{\rho}(h)=\hat{\rho}(g)$. Extending this equality via the analytic representation, we have $\tau_{a}(\hat{\rho}(h))=\tau_{a}(\hat{\rho}(g))$. Then by Lemma \ref{Losing hat},  this equality implies that $\rho(h)=\rho(g)$ in $\End(\Delta)$. Now since $\End(\Delta)\cong \C_{q}(V)$, we can take inverses to conclude that $h=g$, and hence $\hat{\rho}$ is an injective ring homomorphism. To prove surjectivity, choose an arbitrary endomorphism $f\in \End(S_{\Delta})$. Taking its analytic representation gives us $\tau_{a}(f)\in \End(\Delta)$. Now since we have the isomorphism $\C_{q}(V)\cong \End(\Delta)$, there exists an element $g\in\C_{q}(V)$ such that $\rho(g)=\tau_{a}(f)$. Moreover, the analytic representation $\tau_{a}(f)$, when restricted to the full rank lattice $\Gamma_{\Delta}$, is  a $\mathbb{Z}$-linear map on $\Gamma_{\Delta}$. Thus we have $\tau_{a}(f)\big |_{\Gamma_{\Delta}}\in \End_{\mathbb{Z}}(\Gamma_{\Delta})$, implying that $\tau_{a}(f)$ preserves the full rank lattice $\Gamma_{\Delta}\subset \Delta$. 
Now since $\tau_{a}(f)=\rho(g)$, it follows that
 $\rho(g)$ in $\textrm{End}(\Delta)$ is the image of some $g\in\C_{q}(V)$ that has the additional property of preserving the full rank lattice $\Gamma_{\Delta}$, and so we have $g\in\C_{q}(V)_{\mathbb{Z}}\subset\C_{q}(V)$. Thus, since our choice of endomorphism was arbitrary, we have that for any $f\in \End(S_{\Delta})$  there exists an element $g\in \C_{q}(V)_{\mathbb{Z}}$ such that $\hat{\rho}(g)=f$, implying that $\hat{\rho}$ is an isomorphism. 
\end{proof}

From Proposition \ref{endomorphism isomorphism}, we have the understanding that for any endomorphism of $S_{\Delta}$, there exists a lattice element that defines it. Therefore all we need to know, in order to understand the structure of the  endomorphism ring  of our spinor Abelian variety, is to understand the structure of the integer subring $\C_{q}(V)_{\mathbb{Z}}$. We also have the following corollary.

\begin{cor}\label{Automorphisms of spinors}
For our spinor Abelian variety $S_{\Delta}$ we have $\Aut(S_{\Delta})\cong \hat{\Gamma}_{q}^{c}(V)$.
\end{cor}

\begin{proof}
Note that $\Aut(S_{\Delta})$ is the group of units of $\End(S_{\Delta})$, and that by Proposition \ref{endomorphism isomorphism} we have $\C_{q}(V)_{\mathbb{Z}}\cong \End(S_{\Delta})$. Then to find the automorphism group of $S_{\Delta}$, we just need to restrict our attention to the units of the integer subring $\C_{q}(V)_{\mathbb{Z}}$, which are all of the generators $e_{I}$ that generate the real algebra $C_{q}(V)$, and their imaginary generators $ie_{I}$. These generators form the multiplicative group of generators of $\C_{q}(V)$, denoted $\hat{\Gamma}_{q}^{c}(V)$. Hence we have $\hat{\Gamma}^{c}_{q}(V)\cong  \Aut (S_{\Delta})$.
\end{proof}

From Proposition \ref{endomorphism isomorphism} and Corollary \ref{Automorphisms of spinors}, we have a good understanding of the endomorphism ring and automorphism group of our spinor Abelian variety $S_{\Delta}$. Hence we can think of $S_{\Delta}$ as a spinor space for the lattice $\C_{q}(V)_{\mathbb{Z}}$, since $\End(S_{\Delta})\cong \C_{q}(V)_{\mathbb{Z}}$. Knowing the structure of $\C_{q}(V)_{\mathbb{Z}}$ and the multiplicative group of generators provides us with knowledge about the endomorphisms and automorphisms of $S_{\Delta}$. 
\begin{remark}
Another way to see that the multiplicative generators are automorphisms comes from the fact that they preserve the polarization, since they are a subgroup of the $\Pin^{c}(V)$ group, which we know (see \cite[p.\ 74]{Me}) preserves the Hermitian form on our spinor module.
\end{remark}

With respect to intrinsic properties of our spinor Abelian varieties, we  can now prove the following decomposition theorem.
\begin{thm}\label{full decomposition}
A spinor Abelian variety $S_{\Delta}$ is fully decomposable, as a spinor Abelian variety,  as a product of $2^{k}$ copies of the elliptic curve $E_{i}$ of $j$-invariant $1728$.
\end{thm}

\begin{proof}
Let $S_{\Delta}$ be a spinor Abelian variety of dimension $2^{k}$. From Proposition \ref{endomorphism isomorphism} we have that $\End(S_{\Delta})$ is isomorphic as a ring, and hence a free $\mathbb{Z}$-module, to the lattice $\C_{q}(V)_{\mathbb{Z}}$. Thus we can immediately conclude that the rank of the endomorphism ring as a free $\mathbb{Z}$-module is $2^{2k+1}$ (the same as that of $\C_{q}(V)_{\mathbb{Z}}$). Therefore by Proposition \ref{Decomposition} we immediately have that $S_{\Delta}$ is isogenous to the direct sum of $2^{k}$ copies of an elliptic curve with complex multiplication. We next show that this curve is of $j$-invariant $1728$.

By Lemma \ref{Losing hat} and Proposition \ref{endomorphism isomorphism} we have the following commutative diagram:   
\[
\begin{tikzcd}
\C_{q}(V)_{\mathbb{Z}} \arrow[r, "\hat{\rho}"]\arrow[d,"inc"]
& \End(S_{\Delta} )\arrow[d, "\tau_{a}"] \\
\C_{q}(V) \arrow[r, "\rho"]
&  \End(\Delta).
\end{tikzcd}
\]
From Corollary \ref{Automorphisms of spinors} we have the isomorphism  $\Aut(S_{\Delta})\cong \hat{\Gamma}_{q}^{c}(V)$. Hence for the automorphism $\hat{\rho}(i)\in \Aut(S_{\Delta})$ of order $4$, we have $\tau_{a}(\hat{\rho}(i))=\rho(inc(i))$, where $inc:\C_{q}(V)_{\mathbb{Z}}\hookrightarrow \C_{q}(V)$ is the inclusion homomorphism. Thus we have  $\tau_{a}(\hat{\rho}(i))=\rho(inc(i))=\rho(i)=i\cdot \rho(1)=i\cdot id_{\Delta}$. We have shown  that $S_{\Delta}$ has an automorphism of order $4$ whose analytic representation is $i\cdot id_{\Delta}$, and so by Proposition \ref{Lange thm} we have the isomorphism $S_{\Delta}\cong \gls{product of elliptic curves}:= \underbrace{E_{i}\times \ldots \times E_{i}}_{2^{k} \textrm{ times}}$ as  polarized Abelian varieties, where $E_{i}$ is the elliptic curve that admits automorphisms of order $4$; thus it must be of $j$-invariant $1728$. Therefore we have shown that $S_{\Delta}$ is fully decomposable as an Abelian variety. We still have to show that it is fully decomposable as a spinor Abelian variety. Defining the isomorphism $f:S_{\Delta}\xrightarrow {\cong} E_{i}^{\times 2^{k}}$, we can extend Clifford multiplication via $Ad_{f}: \End (S_{\Delta})\rightarrow \End(E_{i}^{\times 2^{k}})$, where $g\mapsto Ad_{f}(g)=f\circ g\circ f^{-1}$. Composing Clifford multiplication with the adjoint conjugation extends Clifford multiplication from $S_{\Delta}$ to $E_{i}^{\times 2^{k}}$, by $\gls{rho^f}:\C_{q}(V)_{\mathbb{Z}}\rightarrow \End(E_{i}^{\times 2^{k}})$, given by $\rho^{f}(h)=Ad_{f}(\hat{\rho}(h))=f\circ \hat{\rho}(h)\circ f^{-1}$ for a given lattice element $h\in \mathbb{C}_{q}(V)_{\mathbb{Z}}$. That is, for any $h\in\C_{q}(V)_{\mathbb{Z}}$ we have the following commutative diagram: 
\[
\begin{tikzcd}
S_{\Delta} \arrow[r, "\hat{\rho}_{h}"]
& S_{\Delta}  \arrow[d, "f"] \\
E_{i}^{\times 2^{k}} \arrow[u, "f^{-1}"] \arrow[r, red, "{\rho}^{f}_{h}" blue]
&  E_{i}^{\times 2^{k}}.
\end{tikzcd}
\]
This shows that we can naturally extend Clifford multiplication onto $E_{i}^{\times 2^{k}}$, making $E_{i}^{\times 2^{k}}$ a spinor Abelian variety.  Hence we have shown that $S_{\Delta}$ is fully decomposable not only as a PPAV, but also as a spinor Abelian variety.

\end{proof}
From Proposition \ref{full decomposition} we have the intrinsic property of $S_{\Delta}$ that all spinor Abelian varieties with Clifford multiplication $\hat{\rho}:\C_{q}(V)_{\mathbb{Z}}\rightarrow \End(S_{\Delta})$ are fully decomposable, as spinor Abelian varieties, as the product of $2^{k}$ copies of the elliptic curve $E_{i}$ of $j$-invariant $1728$. We now have the following immediate corollary when viewing $E_{i}^{\times 2^{k}}$ as a spinor Abelian variety.

\begin{cor}
For the spinor Abelian variety $E_{i}^{\times 2^{k}}$, its endomorphism ring is isomrophic to the integer subring $\C_{q}(V)_{\mathbb{Z}}$, and its group of automorphisms is isomorphic to the multiplicative group of generators of $\C_{q}(V)$. That is, $\End(E_{i}^{\times 2^{k}})\cong \C_{q}(V)_{\mathbb{Z}}$ and $\Aut(E_{i}^{\times 2^{k}})\cong \hat{\Gamma}^c_{q}(V)$.
\end{cor}

\begin{proof}
    This corollary  immediately follows from Propositions \ref{endomorphism isomorphism} and \ref{full decomposition} and Corollary \ref{Automorphisms of spinors}.
\end{proof}

\section{Future work}
We plan to address the following.
\begin{enumerate}

    \item Construct singular curves whose Jacobian is a spinor Abelian variety.
    \item Expand Clifford multiplication to products of elliptic curves of $j$-invariant $1728$. 

    \item Generate more concrete examples of spinor Abelian varieties by focusing on certain types of spinor spaces. Moreover, on strictly real spinor Abelian varieties, we would like to describe what it means for them to be of real, complex, and quaternionic type within this context.
\end{enumerate}

\section{Acknowledgments }
We would like to thank Alfonso Zamora Saiz and George Hitching for providing us with helpful comments and suggestions.

\end{document}